\documentclass[10pt]{amsart}
\usepackage{amscd,amsmath,amssymb,amsthm}
\usepackage{color}
\usepackage{textcomp}
\usepackage{graphicx}

\newtheorem{theo}{{\bfseries Theorem}}[section]
\newtheorem{prop}[theo]{{\bfseries Proposition}}
\newtheorem{thm}[theo]{{\bfseries Theorem}}
\newtheorem{lem}[theo]{{\bfseries Lemma}}
\newtheorem{cor}[theo]{{\bfseries Corollary}}
\newtheorem{df}[theo]{{\bfseries Definition}}

\newtheorem{ex}[theo]{{\bfseries Example}}
\newtheorem{rem}[theo]{{\bfseries Remark}}

\def \cO{\mathcal O}
\def \cN {\mathcal N}
\def \N {\mathbb N}
\def \cP{\mathcal P}
\def \I {\mathcal I}
\def \M {\mathcal M}
\def \K {\mathcal K}
\def \L {\mathcal L}

\def \D {\mathcal D}
\def \X {\mathcal X}
\def \Y {\mathcal Y}
\def \Z {\mathbb Z}
\def \R {\mathbb R}
\def \T {\mathbb T}
\begin{document}
	
\title[Characterization of Quasifactors]{Characterization of Quasifactors}
\vspace{1cm}
\author{ Anima Nagar}
\address{Department of Mathematics, Indian Institute of Technology Delhi, Hauz Khas, New Delhi 110016, INDIA}
\email{anima@maths.iitd.ac.in}
	
\vspace{.2cm}

	\renewcommand{\thefootnote}{}
	
	

	\begin{abstract}
				A flow $(X,T)$ induces the flow $(2^X,T)$. Quasifactors are minimal subsystems of $(2^X, T)$ and hence orbit closures of almost periodic points for $(2^X, T)$. We study quasifactors via the  almost periodic points for $(2^X,T)$.
		
	\end{abstract}
	\maketitle
	
	The two sister branches of  ``Topological Dynamics'' and ``Ergodic Theory'' often have parallel growth with almost similar properties. While \emph{topological dynamics} consists of studying  \emph{flows} $(X,T)$ with $X$  usually a compact topological space and $T$ being the acting	
	topological group, \emph{ergodic theory} deals with \emph{{processes}} $(X,\mu,T)$ on a standard Borel space $X$ with $T$ a measurable 
	transformation satisfying $\mu(T^{-1}(A))= \mu(A)$, for every Borel set $A$.
	
	\bigskip
	
\emph{`Disjointness'} is an important concept in  ``topological dynamics''. This concept was first introduced by Furstenberg for both the topological and ergodic cases \cite{F}, and since has been widely studied.

\bigskip

In \cite{SG} Glasner considered the  induced flow $(2^X,T)$ on the space $2^X$ of non empty closed subsets of $X$, with $T$ induced on $2^X$ and gave a necessary and sufficient condition for two minimal flows to be disjoint. For that he introduced the notion of \emph{`quasifactors'}. Briefly, \emph{quasifactors} are the minimal subsets of $2^X$, and in some way generalize the concept of factors. 

\bigskip

  A closed ally to the concept of disjointness is the concept of \emph{`joinings'}. A vivid exploration of the theory of joinings has been made by Glasner resulting in his fascinating book \cite{JOIN}.

  \bigskip
  
  The processes $(X, \mu, T)$ induce  processes $(\cP(X), \lambda,T)$ where $\cP(X)$ is the space of probability measures on $ X $ equipped with the weak* topology, $\lambda$ the associated measure on $\cP(X)$ and $T$ the induced measure preserving transformation on $\cP(X)$. Motivated by the notion of quasifactors in topological dynamics, Glasner \cite{SGQET} introduced an analogous notion in the context of ergodic theory. A general \emph{ergodic quasifactor}  of $ (X, \mu, T) $ is any $ T $-invariant measure on $ \cP(X) $ whose barycenter is $ \mu $. A joining of two processes gives rise to an ergodic quasifactor for each process, and in fact most ergodic quasifactors are obtained in this way. This defines what are called \emph{joining quasifactors} and have been studied by Glasner and Weiss, see \cite{SGQET, GW, EGBW}.
  
  \bigskip
  
  Auslander \cite{JA} studied the properties of `joining quasifactors' in the domain of topological dynamics via the rich algebraic theory of Ellis groups. The study was further explored and more properties of such `joining quasifactors'  were studied by Glasner \cite{EG}. Various other properties of quasifactors have been explored here, and many interesting examples constructed.

  \bigskip
  
  For a minimal $(X, T)$,  a proper  quasifactor can not be
  disjoint from it. What could be other properties of quasifactors?
  
  Quasifactors may not necessarily inherit all dynamical properties of the system. A weakly mixing process admits a quasifactor which is not weakly mixing. Distality and zero-entropy are preserved by ergodic quasifactors and though distality is preserved, there is a zero-entropy flow that admits positive-entropy quasifactor in the topological case. Ergodic quasifactors are preserved under passage to factors though this is not generally true in the topological case.  In the topological realm, it can be seen that the quasifactor of a minimal equicontinuous system is isomorphic to a factor of the system. Quasifactors of metrizable systems are metrizable and quasifactors of uniformly rigid systems are uniformly rigid. A quasifactor of a minimal proximal system  need not be proximal. We refer to \cite{JA, SG,  SGP, GAPP, SGQET, EG, JOIN, GW, EGBW} for more details.

\bigskip

This leads to a natural motivation of investigating more properties of quasifactors.

\bigskip

 Since quasifactors are minimal subsets of $(2^X, T)$, it becomes important to study the  almost periodic points in $ 2^X $. Note that $2^X$ can never be minimal.

We study quasifactors of flows $(X,T)$ where $X$ is a compact metric space and $T$ is an abelian discrete group, irrespective of the flows being minimal, with the background of  \emph{ enveloping semigroups}. We look into quasifactors as orbit closures of almost periodic points in $2^X$, which happen to be primarily the fixed points of minimal idempotents in the enveloping semigroup $E(2^X)$ or the sets fixed by minimal idempotents in $\beta T$, the Stone-\^{C}ech compactification of $T$, via the circle operator.

It is known that the minimal idempotents in $\beta T$ are  \emph{ultrafilters} mostly comprising of \emph{IP sets}, and all that we know about $E(2^X)$ is the discussion in \cite{TDES}. Thus it is extremely painstaking to compute the idempotents in either of $E(2^X)$ or $\beta T$. On the contrary, it is relatively easier to compute the minimal idempotents in $E(X)$. 

\bigskip

We attempt to compute the almost periodic points in $2^X$ using minimal idempotents in $E(X)$. We derive elementary properties of quasifactors by such a method. After discussing basic theory in Section 1, we discuss the related concept of almost periodic sets in Section 2. Section 3 deals with studying an alternate definition of the circle operator based on $E(X)$. Finally, Section 4 is devoted to the study of almost periodic points in $2^X$ and quasifactors.

\bigskip

Some related stuff is studied in \cite{eAeG}, especially Theorem $1$ and Theorem $1'$ there.

\bigskip

\smallskip

The author thanks Joseph Auslander and Eli Glasner for many insightful discussions.

\section{Preliminaries}

 All our  notations and definitions are as in \cite{TDES}.
 
 \bigskip

Let $(X,d)$ be an infinite, compact metric space and $T$ be a countable, discrete abelian group. $2^X$ is  the space of all nonempty closed subsets of $X$,  endowed with the Hausdorff topology.  

\smallskip

Given a point $x \in X$ and a closed set $A \subseteq X$, recall
$d(x, A) = \inf \limits_{a \in A} d(x, a)$ and  \emph{the Hausdorff metric} is defined as 

$$d_H(A,B) \quad = \quad \max \{\sup \limits_{a \in A} d(a, B),  \sup \limits_{b \in B} d(b, A)\}, \ \forall \ A, B \in 2^X.$$

Since $X$ is compact, we occasionally use an equivalent topology on $2^X$. Define for
any collection $\{ U_i : 1 \leq i \leq n \}$ of  open and nonempty subsets of $X$,

$$\langle U_1, U_2, \ldots U_n \rangle =  \{E \in 2^X :E \subseteq \bigcup \limits_{i=1}^n
U_{i}, \ E \bigcap U_{i} \neq \phi,  \textrm{ } 1 \leq i \leq n \}$$

The topology on $2^X$, generated by such collection as basis, is
known as the {\emph{Vietoris topology}}.

Note that $2^X$ is also compact under this topology.

The induced acting topological group $T$ is defined as: $\forall \ t \in T$, $tA = \{ta:a \in A\}$.

\bigskip
 
 We follow \cite{SG} in calling the transitive flow $(X,T)$ as a \emph{pointed flow} $ (T,X,x_0) $  with a distinguished transitive point $x_0$.

Let $(X,T)$ be a weakly mixing or mixing metric flow. Then there is a $x_0 \in X$ such that $\overline{Tx_0} = X$. Infact, we can write our system as a pointed system $(X,x_0,T)$.

Now the induced flow $(2^X,T)$ is also a weakly mixing(topologically transitive) or mixing metric flow. Then there is a $C_0 \in 2^X$ such that $\overline{TC_0} = 2^X$. Some basic properties of such a transitive point $C_0$ in $2^X$ have been studied in \cite{AN, TDES}.

\bigskip

The flow $(\beta T, T)$ is a universal point transitive
flow, i.e. for every point transitive flow $(X,T)$ and a
point $x_0 \in X$ such that $\overline{{O}(x_0)} = X$, there is a unique homomorphism $\beta T \to X$ such that $e \to x_0$.

Identify $t\in T$ with the map $t\longrightarrow tx$. So without loss of generality $T$ can be considered as a subset of $X^X$. Hence this action of $\beta T$ can be thought of as a representation of $\beta T$ in $X^X$, the compact space of all mappings of $X$ into itself. The image of $\beta T$ there, is a semigroup which is the closure of $T$ (represented itself as a subset of $X^X$) in $X^X$. This semigroup, denoted  by $E(X)$, is the enveloping  semigroup of the flow $(X,T)$. In fact as shown by Ellis,  there exists a continuous map $\Phi: \beta T \rightarrow X^X$ which is an extension of $\phi: T \longrightarrow X^X$ such that $\Phi(\beta T) = E(X)$.

\bigskip

\begin{lem} \label{O(x)} Let $(X,T)$ be a flow and $x \in X$. Then

(i) $\overline{O(x)} = (\beta T)x = E(X)x$

(2) $\overline{O(x)}$ is minimal if and only if $x \in \M x$ for every minimal ideal $\M \subset \beta T \ \text{or} \ E(X)$
 if and only if  for every minimal ideal $\M \subset \beta T \ \text{or} \ E(X)$ there is an idempotent $u \in \M$ such that $ux = x$.\end{lem}

\bigskip

Given a minimal flow $(X,T)$ and an idempotent $u \in \M$, there is a point $x_0 \in uX = \{ux : x \in X\} = \{y \in X: uy = y\}$. Under the canonical map $(\beta T,e) \to (X,x_0)$, $\M$ is mapped onto $X$ and $u$ onto $x_0$. Thus $(\M,u)$ is a universal minimal pointed flow in the sense that for every minimal flow $X$
there is a point $x_0 \in X$ such that $(X,x_0)$ is a factor of $(\M,u)$.

\bigskip

Now $T$ acts on $\beta T$. And $\beta T$ is point transitive, though not necessarily topologically transitive or weakly mixing or mixing. We can also consider the pointed system $(\beta T, e, T)$ where $e$ is the identity in $T$.

Also for $p \in \beta T$ if $t_i \to p$ with $t_i \in T$, we have $px = \lim t_i x$, for all $x \in X$.
The map $\rho: \beta T \to X$ defined as $p \stackrel{\rho}{\to} px_0$ with $p \in \beta T$, and $x_0 \in X$ $-$ the transitive point, defines a flow homomorphism between $(\beta T, e, T)$ and $(X, x_0, T)$.
$$\rho(tq) = tq(x_0) = t(qx_0) =  t \rho(q)$$

We recall  that the minimal subsets of the flow $(\beta T, T)$ (all are isomorphic) coincide with the minimal right ideals of the semigroup $\beta T$.
These are universal minimal flows - every minimal flow is a homomorphic image. We fix a universal minimal flow $(\M, T)$, and let
$J(\M)$ denote the set of idempotents in $\M$. Then $J(\M)$ is non-empty.

Let $(Y, T)$ be a minimal
subflow of $(X,T)$. Then for $y \in Y$, there is a $u \in J(\M)$ such that $uy = y$. Now $ u ty = tu y = ty$. Thus, one can say that $\rho(\M) = Y$. Similarly for every minimal  $Y \subset X$, there exists a minimal right ideal $\M_Y \subset \beta T$ such that $\rho(\M_Y) = Y = \overline{O(y)}$ and an idempotent $u \in J(\M_Y)$ such that $uy = y$.

\vskip .5cm

For the induced flow $(2^X,T)$,  the notion of a ``\emph{circle operator}" as an action of $\beta T$  on $2^X$ was defined by Ellis, Glasner and Shapiro \cite{EGS1}. Identify $t\in T$ with the map $t\longrightarrow tA$, for $A \in 2^X$. So without loss of generality $T$ can be considered as a subset of $(2^X)^{2^X}$. Hence this action of $\beta T$ can be thought of as a representation of $\beta T$ in $(2^X)^{2^X}$, the compact space of all mappings of $2^X$ into itself. The image of $\beta T$ there, is a semigroup which is the closure of $T$ (represented itself as a subset of $(2^X)^{2^X}$) in $(2^X)^{2^X}$. This semigroup is the enveloping  semigroup of the flow $(2^X,T)$. We denote this semigroup by $E(2^X)$. We recall its properties studied in \cite{TDES}.

In fact as shown by Ellis,  there exists a continuous map $\Psi: \beta T \rightarrow (2^X)^{2^X}$ which is an extension of $\psi: T \longrightarrow (2^X)^{2^X}$ such that $\Psi(\beta T) = E(2^X)$.

\bigskip

 $T$ acts on the induced system $2^X$ as $tA=\lbrace ta: a\in A\rbrace$ for any $t\in T$. The \emph{circle operator} is an action of $\beta T$ on the induced system $2^{X}$. Let $\emptyset\neq A= \overline{A}\subset X$ and
 $p\in \beta T$. The circle operation of $\beta T$ on $2^X$ is defined as
 
 \centerline{$p\circ A= \lbrace x\in X: \exists$ nets $ \lbrace a_{i}\rbrace \ \text{in} \ A$,  and  $\lbrace t_{i} \rbrace \ \text{in} \ T$ with $t_{i}\rightarrow p \in \beta T$  such that $t_{i}a_{i}\rightarrow x\rbrace$.}

 \vskip .5cm

Thus for $p \in \beta T$ if $t_i \to p$ with $t_i \in T$, we have $p \circ C = \lim t_i C$, for all $C \in 2^X$, where this limit is taken with respect to the Hausdorff metric. Hence,

\bigskip

\begin{tabular}{|c|}
	\hline \\
$p\circ A= \lbrace x\in X: \exists\ \text{nets} \ \lbrace a_{i}\rbrace$ in $ A, \ \text{and} \ \lbrace t_{i} \rbrace \ \text{in} \ T \ \text{with} \ t_{i}\rightarrow p \in \beta T \ \text{such that}  \ t_{i}a_{i}\rightarrow x\rbrace = \Psi(p)A.$\\
\hline
\end{tabular}

\bigskip

\emph{and the circle operation gives an action of $E(2^X)$ on $2^X$.}

\bigskip

\emph{We note  here that the convergence $t_i \to p$ with $t_i \in T$ does not depend on the topology of $\beta T$, but rather on the action of $T$ on $X$.}

\bigskip

Notice that $pA = \{pa: a \in A\} \subset p \circ A = \Psi(p)A$, since $pa = \lim t_i a$ for $t_i \rightarrow p$. Hence it is also interesting to look into what comprises $p \circ A \setminus pA$. Since $p \circ A$ is the Hausdorff limit of $t_i A$, we can safely presume that $p \circ A \setminus pA = \{b: a_i \to b \ \text{with} \ a_i \in t_iA \ \text{and} \ t_i \to p\ \text{such that} \ b \notin pA\}$. We denote this set by $\widehat{pA}$. It is possible that $\widehat{pA} = \emptyset$, say for example when $A$ is finite.

\bigskip

The map $\zeta: (\beta T, T) \to (2^X,T)$ defined as $p \stackrel{\zeta}{\to} p \circ C_0$ with $p \in \beta T$, and $C_0 \in 2^X$ $-$ the transitive point, defines a flow homomorphism.

\bigskip

$\beta T$ acts on $2^X$ via the circle operator \cite{EGS1}. We denote the cumulative action of $\beta T$ or any $\mathcal{F} \subset \beta T$ on $A \in 2^X$ using the operator $\bigcirc$.

\bigskip

Thus, analogous to Lemma \ref{O(x)} we have

\begin{lem} \label{O(C)} For the flow $(2^X,T)$  and $C \in 2^X$:
	
	(i) $\overline{O(C)} = (\beta T) \bigcirc C = \bigcup  \{p\circ C: p \in \beta T\} = E(2^X)  C = \bigcup \{p C: p \in E(2^X)\}$
	
	(2) $\overline{O(C)}$ is minimal if and only if $C \in \M \bigcirc C$ or $\M C$ for every minimal ideal $\M \subset \beta T \ \text{or} \ E(2^X)$ if and only if in every minimal ideal $\M \subset \beta T \ \text{or} \ E(2^X)$ there is an idempotent $u \in \M$	such that $u \circ C = C$ or $ uC = C$.\end{lem}

\bigskip

Let $(\Y, T)$ be a minimal
subflow of $(2^X,T)$. Then for $B \in \Y$, there is a $u \in J(\M)$ such that $u \circ B = B$. Now $u \circ tB =  ut \circ B = tu \circ B  = tB$ and so $u  \Y \subset u \circ \Y$
is such that $u \circ \Y = \Y$. Thus, one can say that $\zeta(\M) = \Y$. Similarly for every minimal  $\Y \subset X$, there exists a minimal right ideal $\M \subset \beta T$ such that $\zeta(\M) = \Y = \overline{O(B)}$ and an idempotent $u \in J(\M)$ such that $u \circ B = B$.

\bigskip

Since $\beta T$ acts on both $X$ and $2^X$, the action of $p\in \beta T$ on $X$ is given as $x\longrightarrow px$ where $px=\lim t_{i}x$ whereas the action of $p\in \beta T$ on $2^X$ is given as $A\longrightarrow p\circ A$ where $p\circ A = \lbrace x\in X: t_{i}\longrightarrow p, \lbrace a_i \rbrace$ in $ A, t_{i}a_{i}\longrightarrow x\rbrace= \lim t_{i}A $ in $2^X$ where $t_{i}\longrightarrow p$ in $\beta T$.

Now $\Phi(\beta T) = E(X)$ and $\Psi(\beta T) = E(2^X)$. Also $\rho: (\beta T, T) \to (X,T)$ and $\zeta: (\beta T, T) \to (2^X,T)$ are the canonical factors. 

\bigskip

Thus for $t \in T$ and $q \in \beta T$, we can consider the canonical factors:

\begin{center}

 $\Phi: (\beta T, T) \to (E(X),T)$ defined as $\Phi(tq) = t \Phi(q)$, 
 
 $\Psi: (\beta T, T) \to (E(2^X),T)$ defined as $\Psi(tq) = t \Psi(q)$, 
 
 $\varrho: (E(X), T) \to (X,T)$ defined as  $\varrho(t\Phi(q)) =  \rho(tq) = t \rho(q)$ and
 
   $\varsigma: (E(2^X), T) \to (2^X,T)$ defined as  $\varsigma(t\Psi(q)) = \zeta(tq) = t \zeta(q)$.
   
\end{center}

Also we have:

\begin{thm} \cite{TDES} \label{humara}
	For a flow $(X,T)$, there is a continuous flow homomorphism $\theta: (E(2^X), T) \longrightarrow (E(X),T)$: 	where the map $\theta:E(2^X)\rightarrow E(X)$ is defined as 
	\begin{center}
		$\theta(\alpha)= \alpha'$ where $\lbrace\alpha'(x)\rbrace =\alpha(\lbrace x \rbrace)\ \forall\ x\in X$.
	\end{center} \end{thm}

\begin{center}
\begin{tabular}{|c|}
	\hline
Notice that $\theta(\Psi(p)) = \Phi(p), \ \forall \ p \in \beta T$.\\
\hline
\end{tabular}
\end{center}

This gives the following picture:

$$\begin{matrix}
(2^X,T) & & & & (X,T)\\
 & {{\nwarrow {\zeta}}} & &{ {\nearrow{\rho}}} &  \\
 \Bigg\uparrow{\varsigma} & & (\beta T,T) & & \Bigg\uparrow{\varrho}\\
 & { {\swarrow{\Psi}}} & \circlearrowleft &{ {\searrow{\Phi}}} &  \\
(E(2^X),T) & &{ \stackrel{\theta}{\xrightarrow{\hspace*{2cm}}}}&  & (E(X),T)\\
\end{matrix}$$

\vskip .5cm

Again recall,

\begin{theo} \cite{ELL} \label{fac-en} Let $\pi:(X,T)\rightarrow (Y,T)$ be a factor map. Then there exists a factor map $\Pi:E(X)\rightarrow E(Y)$ such that $\Pi(pq)=\Pi(p)\Pi(q)$ and $\pi(px)=\Pi(p)\pi(x)$, $\ \forall \ p,q \in E(X)$ and $x \in X$. \end{theo}

which gives an induced factor map $\pi_*: (2^X, T) \to (2^Y, T)$ and a factor map $\Pi^\dagger: E(2^X) \to E(2^Y)$ such that $\Pi^\dagger(pq)=\Pi^\dagger(p)\Pi^\dagger(q)$ and $\pi_*(pA)=\Pi^\dagger(p)\pi_*(A)$, $\ \forall \ p,q \in E(2^X)$ and $A \in 2^X$.

This gives the following commutative diagram:

$$\begin{matrix}
(X,T) &  { \stackrel{\varrho}{\xleftarrow{\hspace*{1cm}}}} & (E(X),T) & { \stackrel{\theta}{\xleftarrow{\hspace*{1cm}}}} & (E(2^X),T)  &  { \stackrel{\varrho}{\xrightarrow{\hspace*{1cm}}}} & (2^X,T)\\
\Bigg\downarrow{\pi} & \circlearrowleft & \Bigg\downarrow{\Pi} & \circlearrowleft & \Bigg\downarrow{\Pi^\dagger} & \circlearrowleft &  \Bigg\downarrow{\pi_*}\\
(Y,T) &  { \stackrel{\varrho}{\xleftarrow{\hspace*{1cm}}}} & (E(Y),T) & { \stackrel{\theta}{\xleftarrow{\hspace*{1cm}}}} &(E(2^Y),T) & { \stackrel{\varrho}{\xrightarrow{\hspace*{1cm}}}}  & (2^Y,T)\\
\end{matrix}$$

\vskip .5cm

\begin{df} \cite{SG} If $(\Y, T)$ is a minimal subflow of the flow $(2^X, T)$ then we say that $(\Y, T)$	is a quasifactor of (X, T). \end{df}
	
	It is clear that $(X, T)$ itself  as the trivial flow	$(\{X\}, T)$ is a quasifactor.
	
	\vskip .5cm
	
	Now $A \in 2^X$ is an \emph{almost periodic point in $2^X$} if and only if there exists an idempotent $u \in E(2^X)$ for which $u  A = A$.

\vskip .5cm

\begin{tabular}{|c|}
\hline 
Thus the quasi-factors of $(X,T)$ are the orbit closures of essentially the points in $2^X$ fixed by the \\ idempotents in $E(2^X)$ i.e. the almost periodic points in $2^X$. \\ 
\hline
\end{tabular}

\bigskip

 If $(\X , T)$ and $(\Y, T)$ are quasifactors of the flow $(X, T)$, we say that $(\X , T)$ is
finer then $(\Y, T)$ if some element of $\X$ (and hence every element of $\X$ ) is
contained in some element of $\Y$.

\vskip .5cm

The points $x,y \in X$ in the
flow $(X, T)$ are distal if the orbit closure of the point $(x, y)$ in the flow $(X \times X, T)$ does not intersect the diagonal. The flow $(X, T)$ is distal if $ x, y \in X$
and $x \neq y$ implies $x$ and $y$ are distal points. A famous theorem of Furstenberg gives a
description of the structure of a minimal metric distal flow $-$
such a flow is build-up from an equicontinuous flow by successive isometric extensions. 

\bigskip

Let $J(X)$ denote the set of all idempotents in $E(X)$. We recall

\begin{theo} \cite{AF}
	For the system $(X,T)$, let $x \in X$.	Then the following are equivalent:
	
	(i) $x$ is a distal point.
	
	(ii) $vx = x$ for all  $v \in J(X)$.
	
	(iii) $ux = x$ for all minimal idempotents $u \in E(X)$.
\end{theo}

The points  $x, y \in X$ in the flow $(X, T)$ are called proximal if the closure of the orbit of the
point $(x, y)$ in the flow $(X \times X, T)$ intersects the diagonal. The flow $(X, T)$ is
proximal if every two points of $X$ are proximal.

\begin{thm} \cite{SG} If (X, T) is a minimal flow then it has a unique finest
proximal quasifactor.
\end{thm}

\section{Almost Periodic Sets for $(X,T)$}

Recall,

\begin{df}\cite{ AUS}
	For a  flow $(X,T)$, a set $A\subset X$ is said to be an \emph{almost periodic set} if any point $z\in X^{|A|}$ with
	$range(z)=A$ is an almost periodic point of $(X^{|A|},T)$, where $|A|$ is the cardinality of $A$.
\end{df}

The point $z\in X^{|A|}$ is thought of as $A$ ``spread out" to a point in the product space $ X^{|A|}$.

\begin{rem} Also from \cite{AUS} we observe that  for the flow $(X,T)$, and an almost periodic set $ A $
	in $ X $; by Zorn's lemma there is a maximal (with respect to inclusion) almost	periodic set $ B $ such that $ A \subset B $. \end{rem}

For any non-empty index set $\Lambda$ and the product system $(X^\Lambda , T)$ we can identify $\Delta E(X )^\Lambda$ with $E(X^\Lambda)$. Thus, we have for any $k$-tuple $(x_1, x_2, \ldots , x_k) \in X^k$ - the
pointed system 

$(\overline{T(x_1, x_2, \ldots , x_k)}, (x_1, x_2, \ldots , x_k), T)$ is a factor of $(E(X)^k, (\underbrace{e,e,\ldots,e)}_{k- \text{times}}, T)$.

\bigskip

This gives another characterization of almost periodic sets which can be taken as another definition:

Let $A\subset X$  be an \emph{almost periodic set}. Then  $z\in X^{|A|}$ with
$range(z)=A$ is an almost periodic point of $(X^{|A|},T)$, where $|A|$ is the cardinality of $A$. This gives a minimal idempotent  $\textbf{u} \in E(X^{|A|})$ such that $\textbf{u} z = z$. But $\textbf{u} = \underbrace{(u,u, \ldots)}_{|A|-times}$, and so we have $ua = a$, $\forall a \in A$. Thus,

\begin{lem} For the flow $(X,T)$, the following are equivalent:
	
	1. $A \subset X$ is an almost periodic set.
	
	2. There exists a minimal idempotent $u \in E(X)$ such that $ua = a, \ \forall a \in A$. \end{lem}

\bigskip

 For each $u \in J(X)$, consider the set 

$$F_{u}=\lbrace x\in X: ux=x \rbrace$$

 Let $x\in X$  and $u \in J(X)$, then $uux=ux$, so $ux\in F_{u}$. So $F_{u} = uX$ is always non-empty. But there is more to $F_u$.
 
 \bigskip
 
 Recall the concept of quasi-order on $J(X)$ given in \cite{AF}. A \emph{quasi order} (a reflexive, transitive relation) ` $>$ ' in $J(X)$ is defined as
$ u > v$ if $uv = v$. If $u > v$ and $v > u$ we say that $u$ and $v$ are \emph{equivalent}   and write $u \sim v$. 

 An $w \in J(X)$ is called maximal if whenever $u \in J(X)$ with $u > w$, then
$w \sim u$. Minimal idempotents are defined similarly.  With respect to the quasi order ` $>$ ' $J(X)$ contains maximal and minimal
idempotents. If $v \in J(X)$, there are maximal and minimal idempotents $w$ and $u$ such that $w > v > u$. 

\vskip .5cm

Since our $T$ is a group, the maximal idempotent in $J(X)$ is $e \ -$ the identity in $T$.

\vskip .5cm

In fact,  the minimal idempotents are precisely those idempotents
which are in some minimal left ideal of $E(X)$. Let $ \K, \L $ and $ \M $ be minimal ideals of $E(X)$. Let $ v \in J(\M)$. Then there is a unique idempotent
$ v' \in J(\L) $ such that $ vv' = v' $ and $ v'v = v $. We 
say that $ v' $ is \emph{equivalent} to $ v $ and write $ v \sim v' $. If $ \acute{v} \in J(\K) $ is equivalent to $ v' $ then $  v \sim \acute{v} $. The map $  p \to pv' $ of $ \M $ onto $ \L $ is an isomorphism of flows.

\vskip .5cm

\begin{lem}
	Let $u,v \in J(X)$ be  such that $u > v$. Then $F_u \supset F_v$. 
	
	Further if $u \sim v$, then $F_u = F_v$.
\end{lem}
\begin{proof}
	The proof follows trivially since for $u > v$, $uvx = vx$ for all $x \in X$. 

	Also if $u \sim v$ then $vux = ux $ and $uvy = vy$ for all $x,y \in X$. 
\end{proof}

\begin{prop}
	For a flow $(X,T)$ let $u\in E(X)$ be a minimal idempotent. Then the set $F_{u}$ is an almost periodic set.
\end{prop}
\begin{proof}
	Let $z\in X^{|F_{u}|}$ be any point such that $range(z)=F_{u}$.   Consider $\textbf{u}= (u, u, u, \ldots)\in \Delta E(X)^{|F_{u}|} \cong E(X^{|F_{u}|})$. Since $range(z)=F_{u}$, $\textbf{u}z=z$, i.e. $z$ is an almost periodic point in $X^{|F_{u}|}$. Hence $F_{u}$ is an almost periodic set.  
\end{proof}  

\begin{cor}
	Every almost periodic set is contained in some $F_{u}$ for some minimal idempotent $u\in E(X)$. In particular, every almost periodic set is of the form $uA$ for some $A \subseteq X$.
\end{cor}

\begin{cor}
	Every maximal almost periodic set is of the form $F_{u}$ for some minimal idempotent $u\in E(X)$.
\end{cor}

\begin{prop}
	For a flow $(X,T)$, let $u \in E(X)$ be a minimal idempotent, and let $x\in X$. Then there is an $x^{\prime}\in F_u \cap \overline{\cO(x)}$ such that $x$ and $x^{\prime}$ are proximal.
\end{prop}
\begin{proof} Since $ux \in uX = F_u$, let $x' = ux$. Then, $x^{\prime}\in F_u \cap \overline{\cO(x)}$ and   $ux = x' = ux'$. Thus, $x$ and $x'$ are proximal.
\end{proof}

We recall Theorem \ref{fac-en}, using which we have:

\begin{prop}
	Let $\pi: (X,T)\longrightarrow (Y,T)$ be a factor map and $u \in E(X)$ be a minimal idempotent. Then $\Pi(u) \in E(Y)$ is a minimal idempotent with $F_{\Pi(u)} = \Pi(u)Y \subset Y$ such that $\pi(F_u) = F_{\Pi(u)}$.
\end{prop}
\begin{proof}
	It is known that for idempotent $u \in E(X)$, $\Pi(u) \in E(Y)$ will also be an idempotent \cite{ELL}. Thus, by Thoerem \ref{fac-en} it follows that $\pi(F_u) = F_{\Pi(u)}$.
\end{proof}

\begin{cor}
	Let $\pi: (X,T)\longrightarrow (Y,T)$ be a factor map, $u \in E(X)$ be a minimal idempotent and $A \subset F_u  \subset X$. Then  $A = uA$ is an  almost periodic set in $X$, and $\pi(A)$ is an  almost periodic set in $Y$ with $\pi(A) = \Pi(u)\pi(A)$ for $\Pi(u) \in E(Y)$.
\end{cor}

\vskip .5cm

Suppose there is a distal point $x\in X$.  So $ux=x$ for all  idempotents $u$. So for any minimal idempotent $u$, $F_{u}\neq \emptyset$ and $x\in \bigcap F_{u}$. 

This leads to an interesting observation $-$ If every $x \in X$ is distal i.e. the system $(X,T)$ is distal, then   $X = \bigcap \limits_u F_{u}$ thus $x = ux$ for all $x \in X$ and every  idempotent $u$. This means that all  idempotents coincide on $X$ and so $E(X)$ has a unique  idempotent. This further gives a unique minimal ideal in $E(X)$ implying that $E(X)$ itself is minimal and this unique idempotent must be identity - thus proving \emph{Ellis' theorem}. 

\bigskip

We look into this study via an example:

\begin{ex} \label{MT}
	We look into a substitution system,  the square of the Morse-Thue substitution as considered in \cite{HJ, S}. This is a continuous substitution $Q$ defined by the rule
	$$Q(0) = 0110, \ Q(1) = 1001.$$
	
	We have four bi-infinite sequences that serve as fixed points of $Q$,
	
	\begin{align*}
	a = \ldots 1001.1001 \ldots \\
	b = \ldots 0110.1001 \ldots \\
	\bar{a} = \ldots 0110.0110 \ldots\\
	\bar{b} = \ldots 1001.0110 \ldots \\
	\end{align*}

where $\bar{y}$ denotes the \emph{dual} of $y$ in $\{0,1\}^{\Z}$.

If $x$ denotes any one of the fixed points of $Q$, then $X = \overline{\cO(x)}$ can be defined uniquely for any $x \in \{a,b,\bar{a}, \bar{b}\}$. The system $(X, \sigma)$ is a minimal subsystem of the \emph{$\emph{2}-$shift}.

Following the calculations in \cite{HJ, S}, here $E(X)$ has exactly four minimal  idempotents $u_1, v_1, u_2, v_2$ such that they  are the identity off the orbits of $a,b,\bar{a}, \bar{b}$ and on the orbits of $a,b,\bar{a}, \bar{b}$ are defined as:

\begin{align*}
u_1: a \to b ; \ \bar{a} \to \bar{b} ; \ b \to b ; \ \bar{b} \to \bar{b}\\
v_1: a \to \bar{b} ; \ \bar{a} \to b ; \ b \to b ; \ \bar{b} \to \bar{b}\\
u_2:  a \to a ; \ \bar{a} \to \bar{a} ; \ b \to a ; \ \bar{b} \to \bar{a}\\
v_2: a \to a ; \ \bar{a} \to \bar{a} ; \ b \to \bar{a} ; \ \bar{b} \to a\\ 
\end{align*}

Note that, $u_1v_1 = v_1$, $v_1u_1 = u_1$, $u_2v_2 = v_2$, and $v_2u_2 = u_2$. Thus, $u_1 \sim v_1$ and $u_2 \sim v_2$, and $E(X)$ here has two minimal ideals $\I, \K$ such that $u_1, u_2 \in \I$ and $v_1, v_2 \in \K$. Notice,

\begin{align*}
a, \bar{a} \notin F_{u_1} = F_{v_1} \ni b, \bar{b}\\
 b, \bar{b} \notin F_{u_2} = F_{v_2} \ni a, \bar{a}\\
\end{align*}
	
and $F_{u_1} \cap F_{u_2}$ consists of all the points off the orbits of $a,b,\bar{a}, \bar{b}$, which are precisely the distal points in $(X,\sigma)$.\end{ex}

\begin{rem} \label{remark}
	We note that for $u \sim v \in E(X)$ and $A \in 2^X$, $uA$ need not be equal to $vA$. As can be seen in the example above where $u_1 \sim v_1$ and for $A = \{a,b\} \in 2^X$, $u_1(A) = \{b\} \neq \{b,\bar{b}\} = v_1(A)$.
\end{rem}

What happens when the flow is proximal? As noted in \cite{SGP}, $ P(X) = X \times X$, and so $P(X)$ is an equivalence relation and hence $E(X)$ contains a unique minimal ideal $ \I $. For any $x \in X$ and $u \in J(\I)$, $x = ux$ i.e. $F_u = X$ for all minimal idempotents $u \in E(X)$.

\section{ Prolongations on $2^X $ and Circle Operator }

We recall Auslander's \emph{prolongation relation} $\D$ on $X \times X$:

$$ (x,y) \in \D \Leftrightarrow y \in \D(x)$$

where $\D(x)$ can be thought of as a function $ X \to 2^X$ defined as

\begin{align*}
\D(x) = \bigcap \{\overline{TU}: U \subset X \  \  \text{open with} \  \ x \in U \}.\\
= \lbrace y \in X: \exists\ \text{nets} \ \lbrace x_{i}\rbrace \ \text{with} \ x_i \to x, \ \text{and} \ \lbrace t_{i} \rbrace \ \text{in} \ T \  \ \text{such that}  \ t_{i}x_{i}\rightarrow y\rbrace.\\
\end{align*}

The relation $\D = \overline{\{(x,tx): x \in X, t \in T\}}$ is known to be closed, reflexive and symmetric. And for $A \subset X$, $\D(A) = \bigcup \limits_{a \in A} \D(a)$, and $\D(tA) = \D(A) \ \forall \ t \in T$.

Note that $\D(x) \neq \overline{Tx}$ always, and $\D$ is not always an equivalence relation. Conditions when both these assertions hold are studied in \cite{eAeG}.

\bigskip

By interchanging the role of nets in $X$ and $T$, we consider a variation of this relation:

\smallskip

 For $\alpha \in X^X$ define
 
 \begin{center}
 	\begin{tabular}{|c|}
 		\hline
 		$\D_\alpha(x) = \lbrace y \in X: \exists\ \text{nets} \ \lbrace x_{i}\rbrace \ \text{with} \ x_i \to x, \ \text{and} \ \lbrace t_{i} \rbrace \ \text{in} \ T \ \text{with} \ t_i \to \alpha  \ \text{such that}  \ t_{i}x_{i}\rightarrow y\rbrace.$\\
 		\hline
 	\end{tabular}
 \end{center}

\smallskip

We note that for $\alpha \notin E(X)$, $\D_\alpha(x) = \emptyset \ \forall \ x \in X$. Thus, $\D_p(x) \neq \emptyset \ \forall \ x \in X$, if and only if $p  \in E(X)$. 

\medskip

Note that for $p \in E(X)$,  $px \in \D_p(x), \ \forall \ x \in X$. This $\D_p$ is also a function from $ X  \to 2^X \ \forall \ p \in E(X)$.  We call this function  $\D_p$ - \emph{prolongation along $p$}.

\begin{thm} \label{3.1}
	$\bigcup \limits_{p \in E(X)} \D_p(x) = \ \D(x), \ \forall \ x \in X$.
\end{thm}

\begin{proof}
	It is a simple observation that $\D_p(x) \subset \D(x) \ \forall p \in E(X)$. The converse follows with the observation that for $y \in \D(x)$ and the net $\lbrace t_n \rbrace$ in $T$, if necessary by passing to a subnet, there will exist a $p \in E(X)$ for which  $t_n \to p$. So for net $\lbrace x_n \rbrace$ for which $x_n \to x$ and a net $\lbrace t_n \rbrace$ in $T$,  $t_nx_n \to y$ with $t_n \to p$. Thus, $y \in \D_p(x)  \subset \bigcup \limits_{p \in E(X)} \D_p(x)$. 	\end{proof}

\bigskip

We skip the trivial proof of the below lemma:

\begin{lem}
	For  $t \in T$ and $p, q \in E(X)$ the following holds:
	
	1. $\D_p(tx) = t \D_p(x)$.
	
	2. $\D_{tp}(x) = t \D_p(x)$.
\end{lem}

Again for equicontinuous $(X,T)$, $ E(X)$ is a topological group and so $t_n \to p \Leftrightarrow t_n^{-1} \to p^{-1}$, which gives 

\begin{lem}
	For equicontinuous $(X,T)$, $y \in \D_p(x) \Leftrightarrow x \in \D_{p^{-1}}(y)$.
\end{lem}

\bigskip

We are mainly interested in the same \emph{prolongation along $p$}, $\D_p$ defined on  $2^X$, as a variation of the definition of prolongation. For $p \in E(X)$ and $A \in 2^X$ define,

\bigskip

\begin{tabular}{|c|}
	\hline
	$\D_p(A) = \lbrace y \in X: \exists\ \text{nets} \ \lbrace a_{i}\rbrace \ \text{in} \ A, \ \text{and} \ \lbrace t_{i} \rbrace \ \text{in} \ T \ \text{with} \ t_i \to p  \ \text{such that}  \ t_{i}a_{i}\rightarrow y\rbrace.$\\
	\hline
\end{tabular}

\smallskip

\begin{prop}
	 For every $p\in E(X)$, $\D_p(A)$ is closed $\forall A \in 2^X$ and so $\D_p: 2^X \to 2^X$ is a function.
\end{prop}
\begin{proof}
	Let $z \in \overline{\D_p(A)}$ for some $A \in 2^X$, and let $\{B_\lambda\}$ be a neighbourhood base at $z$. Then, for each $\lambda$, $B_\lambda \cap \D_p(A) \neq \emptyset$. Hence there exists a neighbourhood base $\{\cN_\lambda\}$ at $p$ such that $t_\lambda a_\lambda \in B_\lambda$ for some $t_\lambda \in \cN_\lambda$ and $a_\lambda \in A$.
	
	Then  for $\{a_\lambda\}$ in $A$ and  $\{t_\lambda\}$ in $T$, $t_\lambda \to p$, with $t_\lambda a_\lambda \to z$ i.e. $z \in \D_p(A)$.
\end{proof}

Since $tX = X$ for every $t \in T$, we have vacuously,

\begin{lem} \label{X}
	For every $p \in E(X)$, $\D_p(X) = X$.
\end{lem}

\begin{lem}
	Let $v \in E(X)$ be an idempotent, then for $A \in 2^X$, $vA \subset \D_v(A)$ i.e. $F_v \cap \D_v(A) \neq \emptyset$.
\end{lem}

\begin{rem} \label{contain} Recall Example \ref{MT} and the points $a,b,\bar{a}, \bar{b} \in X$ there. Also recall the minimal idempotents $u_1, v_1, u_2, v_2 \in E(X)$ there and that $a, \bar{a} \notin F_{u_1}$.  But by Lemma \ref{X}, $\D_{u_1}(X) = X$ and so the inclusion in the above Corollary is usually strict.
	
	Note that for nondistal $(X,T)$, $uX \neq X$ for minimal idempotent $u \in E(X)$ and so in this case $uX \neq \D_u(X)$.
\end{rem}

\begin{lem}
	For $t \in T$ and $A \in 2^X$, $\D_t(A) = tA$.
\end{lem}

\begin{proof}
	Note that $tA \subset \D_t(A)$ trivially. For the converse, observe that for $y \in \D_t(A)$ we have $t_\alpha \to t$ for net $\lbrace t_\alpha \rbrace$ in $T$ and a net $\lbrace a_\alpha \rbrace$ in $A$ with $t_\alpha a_\alpha \to y$. But $t_\alpha A \to t A$ and so $y \in tA$. 
\end{proof}

\begin{rem}
	It can be seen that, $\D_p(\{x\})  = \{px\}  \subset \D_p(x)$. 
	
	Also, $\bigcup \limits_{x \in A} \D_p(x) \subset \D_p(A), \ \forall \ p \in E(X)$ and  $  \D(A) \subset \bigcup \limits_{p \in E(X)} \D_p(A), \ \forall \ A \in 2^X$ with both these inclusions  usually strict.

\bigskip

	For each $p \in E(X)$, $\D_p: 2^X \to 2^X $ is a closed function, though it need not be continuous.
\end{rem}

\begin{lem} \label{rules}
	For  $A \in 2^X$, $t \in T$ and $p, q \in E(X)$ the following holds:
	
	1. $\D_p(tA) = t \D_p(A)$.
	
	2. $\D_{tp}(A) = t \D_p(A)$.
	
	3. $\D_{pq} (A) = \D_p (\D_q(A)) $.

\end{lem}

\begin{proof}
	We skip the trivial proofs of 1. and 2.

\smallskip	
	
		For 3. we see that 
	
	\begin{align*}
	\D_p(\D_q(A)) = \D_p ( \lbrace y \in X: \exists\ \text{nets} \ \lbrace a_{i}\rbrace \ \text{in} \ A, \ \text{and} \ \lbrace t_{i} \rbrace \ \text{in} \ T \ \text{with} \ t_i \to q  \ \text{such that}  \ t_{i}a_{i}\rightarrow y\rbrace)\\
	=  \lbrace z \in X: \exists\ \text{nets} \ \lbrace y_{j}\rbrace \ \text{in} \ \D_q(A), \ \text{and} \ \lbrace t_{j} \rbrace \ \text{in} \ T \ \text{with} \ t_j \to p  \ \text{such that}  \ t_{j}y_{j}\rightarrow z\rbrace\\
	=  \lbrace z \in X: \exists\ \text{nets} \ \lbrace a_{ij}\rbrace \ \text{in} \ A, \ \text{and} \ \lbrace t_{ij} \rbrace \ \text{in} \ T \ \text{with} \ t_{ij} \to pq \ \text{such that}  \ t_{ij}a_{ij}\rightarrow z\rbrace\\
	= \D_{pq} (A) \ [\text{since multiplication on right is continuous in semigroup $E(X)$}]
	\end{align*} \end{proof}

\begin{cor}
	For $ \ \forall \ p \in E(X)$,  $\D_{p^n}  = (\D_p)^n  $ on $2^X$ for all $n \in \N$ .
\end{cor}

\begin{cor} \label{idempotent}
	Let $u \in E(X)$ be an idempotent, then $(\D_u)^n = \D_u$ on $2^X$ for all $n \in \N$ .
\end{cor}

Thus $\D_u$ is an idempotent function for every idempotent $u \in E(X)$.

\begin{cor}
	For the system $(X,T)$, the collection $\{\D_p: p \in E(X)\}$ of  self-maps on $2^X$ is a monoid in $(2^X)^{2^X}$ with $\D_e$ being an identity.
\end{cor}

Note that,  this monoid need not be abelian since it is possible that $pq \neq qp$  in $E(X)$.

\begin{lem}
	For each   $p \in E(X)$ and a finite $A \in 2^X$, $\D_p(A) = pA$.
\end{lem}

\begin{rem}
	Recall Remark \ref{remark}, we note that for $u \sim v \in E(X)$, $\D_u(A)$ need not be equal to $\D_v(A)$ for some $A \in 2^X$.
\end{rem}

\begin{prop}
	For flows $(X, T)$ and $(Y, T)$,  let $\pi : (X, T) \to (Y, T)$ be a conjugacy. Then $\pi(\D_p(A)) = \D_{\Pi(p)}(\pi(A)) = \D_{\Pi(p)}(\pi_*(A))$, for all $p \in E(X)$ and $A \in 2^X$.
	
	Further, if $\pi$ is a factor map then $\pi^{-1}(\D_{\Pi(p)}(B)) = \D_{p}(\pi^{-1}(B))$, for all $p \in E(X)$ and $B \in 2^Y$.
\end{prop}

\begin{proof} Note that for conjugacy $\pi : (X, T) \to (Y, T)$, the induced map $\pi_* : (2^X, T) \to (2^Y, T)$ is also a conjugacy. Let $\pi(A) = B = \pi_*(A)$. 
	
	And note that  $p \in E(X) \Longrightarrow \Pi(p) \in E(Y)$ and $t_i \to p \stackrel{\Pi}{\xrightarrow{\hspace*{1cm}}} t_i \to \Pi(p)$.

Hence, 

	$\pi(\D_p(A)) = \pi( \lbrace y \in X: \exists\ \text{nets} \ \lbrace a_{i}\rbrace \ \text{in} \ A, \ \text{and} \ \lbrace t_{i} \rbrace \ \text{in} \ T \ \text{with} \ t_i \to p  \ \text{such that}  \ t_{i}a_{i} \rightarrow y \rbrace)$
	
	 $=  \lbrace \pi(y) \in Y: \exists\ \text{nets} \ \lbrace \pi(a_{i})\rbrace \ \text{in} \ B, \ \text{and} \ \lbrace t_{i} \rbrace \ \text{in} \ T \ \text{with} \ t_i \to \Pi(p)  \ \text{such that}  \ t_{i} \pi(a_{i}) \rightarrow \pi(y)\rbrace$
	 
	 $= \D_{\Pi(p)}(B) = \D_{\Pi(p)}(\pi_*(A))$.
	 
	 \medskip
	 
	 The second part follows by assuming $A = \pi^{-1}(B)$.	
	
\end{proof}

\begin{lem}
	For equicontinuous $(2^X,T)$, $B \in \D_p(A) \Leftrightarrow A \in \D_{p^{-1}}(B)$.
\end{lem}

\begin{proof}
	Recall that here $E(X) \cong E(2^X) -$ a group of homeomorphisms  \cite{TDES}, and so $t_\lambda \to p \Leftrightarrow t_\lambda^{-1} \to p^{-1}$.
\end{proof}

\bigskip

On the lines of Theorem \ref{3.1} we define $\D_* -$  \emph{ the prolongation } on $2^X$ as $\D_*(A) = \bigcup \limits_{p \in E(X)} \D_p(A), \ \forall \ A \in 2^X$.

Since compact union of closed sets is closed, we see that $\D_*: 2^X \to 2^X$ is actually a function. Note that $\D_*$ is a closed function though it need not be continuous.

\begin{lem}
	For the flow $(2^X,T)$, we have the following for the prolongation $\D_*$:
	\begin{enumerate}
		\item $\D_*(A)$ is $T-$invariant, i.e. $\D_*(tA)= \D_*(A) \ \forall \ t \in T$.
		
		\item Let $(2^Y,T)$ be another induced flow with $\pi_* : (2^X, T) \to (2^Y, T)$  a conjugacy. Then $\pi_*(\D_*(A)) = \D_*(\pi_*(A)) \ \forall  \ A \in 2^X$, where we denote the prolongation on both $2^X$ and $2^Y$ by $\D_*$.
		
		Further, if $\pi_*$ is a factor map then $\pi_*^{-1}(\D_*(B)) = \D_*(\pi_*^{-1}(B)) \ \forall B \in 2^Y$.
	\end{enumerate}
\end{lem}

We skip the trivial proof.

\bigskip

For $p \in \beta T$ and $A \in 2^X$, we recall the \emph{circle operator},
$$p\circ A= \lbrace x\in X: \exists\ \text{nets} \ \lbrace a_{i}\rbrace \ \text{in} \ A, \ \text{and} \ \lbrace t_{i} \rbrace \ \text{in} \ T \ \text{with} \ t_{i}\rightarrow p  \ \text{such that}  \ t_{i}a_{i}\rightarrow x\rbrace.$$

Note that,  $p \in \beta T \Rightarrow \Psi(p) \in E(2^X) \Rightarrow \theta(\Psi(p)) = \Phi(p) \in E(X)$ and thus:

\begin{center}

	\begin{tabular}{|c|}
		\hline
		
		$p \circ A = \Psi(p)(A) = \D_{\Phi(p)}(A)$\\ 
		\hline
	\end{tabular}
	
\end{center}

Thus,   $\overline{O(A)} =  E(2^X) A =   \bigcup \{\Psi(p)A: p \in \beta T\} = \bigcup \{p \circ A: p \in \beta T\} =\beta T \bigcirc A$.

\bigskip

This gives, for $A \in 2^X$ and $p,q \in \beta T$:

\begin{enumerate}
	\item $t \circ A = \D_t(A) = tA$.
	\item ${\Phi(p)}A \subset \D_{\Phi(p)}(A) = \Psi(p) A = p \circ A$, and this containment is usually strict as mentioned in Remark \ref{contain}.
	\item $pq \circ A = \D_{\Phi(pq)}(A) = \D_{\Phi(p)} \D_{\Phi(q)}(A) = (p \circ q) \circ A$.
	\item $ \D_*(A) = \overline{O(A)} =  \beta T \bigcirc A =  \bigcup \{p\circ A: p \in \beta T\} = \bigcup \{\Psi(p) A: \Psi(p) \in E(2^X)\} =\bigcup \limits_{\Phi(p) \in E(X)} \D_{\Phi(p)}(A)$. 
\end{enumerate}

\bigskip

\section{ Quasifactors and Almost Periodic Points for $(2^X, T)$}

Quasifactors are minimal subsystems of $(2^X, T)$ and hence orbit closures of almost periodic points for $(2^X, T)$. Thus in order to understand the characteristics of quasifactors, one  needs to isolate the properties of the almost periodic points of $(2^X, T)$.

Note that $A \in 2^X$ is an almost periodic point for $(2^X, T)$ if $\tilde{u} \circ A = A$ for some minimal idempotent $\tilde{u} \in \beta T$, i.e.  $\Psi(\tilde{u})(A) = A$ for the minimal idempotent $ \Psi(\tilde{u}) \in E(2^X)$, i.e.  $\D_u(A) = A$ for the minimal idempotent $u  = \Phi(\tilde{u}) \in E(X)$.

\bigskip

Which elements of $2^X$ are almost periodic? What is the general nature of these sets?

\bigskip

 Recall Corollary \ref{idempotent}, for every minimal idempotent $u \in E(X)$ we notice that $\D_u(\D_u)(A) = \D_u(A) \ \forall \ A \in 2^X$. This gives,
 
 \begin{theo}
 	The almost periodic points for $(2^X, T)$ are precisely the elements of $2^X$ in the range of $\D_u$, for every minimal idempotent $u  = \theta(\bar{u}) = \Phi(\tilde{u}) \in E(X)$, for all minimal idempotents $\bar{u} \in E(2^X)$ or $\tilde{u} \in \beta T$.
 \end{theo}
 
 Recall Theorem \ref{humara}. Now, $\theta: J(2^X) \to J(X)$ need not be surjective \cite{TDES}. Thus, the almost periodic points in $2^X$ are precisely the elements in the range of $\D_u$ when $u  = \theta(\bar{u})$, for $\bar{u}
\in J(2^X)$.

\bigskip 
 
Since $uA \subset \D_u(A)$, $F_u \cap \D_u(A) \neq \emptyset \ \forall \ A\in 2^X$. Hence each almost periodic element in $2^X$ contains almost periodic points of $X$.

 \begin{lem} A finite $A \in 2^X$ is an almost periodic point in $(2^X, T)$ if and only if it is an almost periodic set for $(X,T)$. \end{lem}

Also $\overline{uA} \subset \D_u(A) (\in 2^X)$ for every $A \in 2^X$ and minimal idempotent $u \in E(X) $. And for finite $A \in 2^X$, $\overline{uA} = \D_u(A)$. Thus on a dense set in $2^X$, we have $d_H(\overline{uA}, \D_u(A)) = 0$. Hence, $\overline{uA}$ approximates $\D_u(A)$ to a fairly large extent.

\bigskip

\begin{lem}
	If $(x,y) \in P(X)$ with  $x \in F_u \cap A$ for some $A \in 2^X$, and some minimal idempotent $u \in E(X)$ then $\D_u(\{y\} \cup A) = \D_u(A)$
\end{lem}
\begin{proof}

If $(x,y) \in P(X)$ with  $x \in F_u \cap A$, then since $ ux = uy$ we have $ uy = ux \in \D_u(A) $. Clearly, $\D_u(\{y\} \cup A) \supset \D_u(A)$.

Since $ \D_u(\{y\} \cup A) =  \lbrace z \in X: \exists\ \text{nets} \ \lbrace a_{i}\rbrace \ \text{in} \ A \cup \{y\}, \ \text{and} \ \lbrace t_{i} \rbrace \ \text{in} \ T \ \text{with} \ t_i \to u  \ \text{such that}  \ t_{i}a_{i} \rightarrow z \rbrace$, the net $\lbrace a_{i}\rbrace$ either has a subnet of the constant term $y$ or else a subnet entirely of elements of $A$. In either case the resulting subnet $\{t_{i}a_{i}\}$ will converge in  $\{uy\} \cup \D_u(A) = \D_u(A)$. 
\end{proof}

\begin{cor}
	If $(A,B) \in P(2^X)$, then there exists a minimal idempotent $u \in E(2^X)$ such that $uA = uB$ i.e. $\D_{\theta(u)}(A) = \D_{\theta(u)}(B)$. Thus $\D_u$ need not be injective.
\end{cor}

\bigskip

What happens when the system is distal? We recall the below theorem from \cite{SG} and  give an alternate proof of this, using functions $\D_p$ for $p \in E(X)$ by constructing such a factor.

\begin{theo} \cite{SG}
	Let $(X, T)$ be a minimal distal flow, then every quasifactor of $(X, T)$ is a factor of $(E(X), T)$.
\end{theo}

\begin{proof}
Let $(\X,T)$ be a quasifactor. Then there exists an $A \in 2^X$ with a $u \in (\Phi^{-1}e)$ such that $u \circ A = \D_e(A) = A$ for the only minimal idempotent $e \in E(X)$, and $\X = \overline{\cO(A)} =\bigcup \limits_{p \in E(X)} \D_p(A)$.

 Define $\Gamma: E(X) \to \X$ as 
$$\Gamma(q) = \D_q(A)$$

We observe that:

\begin{enumerate}
	
\item  Clearly $\Gamma$ is surjective.

\item  To prove that $\Gamma$ is continuous, we will show that $p_\nu \to p \Rightarrow \D_{p_\nu}(A) \to \D_p(A)$. Since $(X,T)$ is minimal, distal so is $E(X)$.

Let $U \subset X$ be any open set such that $\D_p(A) \subset U$. Consider the subbasic open set $[a,U] = \{f \in X^X: f(a) \in U\}$ for every $   a \in A$.

Then $p \in \bigcap \limits_{a \in A} [a,U] = \mathcal{U}$. Since $\mathcal{U}$ is open  there is a $d$ in the directed set associated with the net $\{p_\nu\}$ with $p_\nu \in \mathcal{U}$ for all $\nu \geq d$ in this directed set.

Now $\D_{p_\nu}(A) =  \lbrace y \in X: \exists\ \text{nets} \ \lbrace a_{\nu_i}\rbrace \ \text{in} \ A, \ \text{and} \ \lbrace t_{\nu_i} \rbrace \ \text{in} \ T \ \text{with} \ t_{\nu_i} \to p_\nu  \ \text{such that}  \ t_{\nu_i}a_{\nu_i} \rightarrow y \rbrace$. Since $p_\nu \in \mathcal{U} \ \forall \nu \geq d$ there exists a $c_\nu$ in the directed set associated with $\{t_{\nu_i}\}$ such that $t_{\nu_i} \in \mathcal{U}$ for all $\nu_i \geq c_\nu$ in the associated directed set, implying that $\D_{p_\nu}(A) \subset U$. 

Thus, $\D_{p_\nu}(A) \to \D_p(A)$ in $2^X$ i.e. $\Gamma $ is continuous.

\item  Note that $\Gamma(tp) = \D_{tp}(A) =  t\D_{p}(A) = t \Gamma(p)$.
\end{enumerate}

Thus $ (E(X),T ) \stackrel{\Gamma}{\to} (\X,T)$ gives a factor.\end{proof}

\begin{cor}
	A quasifactor of a distal, minimal flow is distal.
\end{cor}

We consider examples of some minimal systems:

\begin{ex} We consider the example  first discussed by Furstenberg \cite{Fu}.
	
	Let $\T = \R/\Z$ be the one-torus, and let $\alpha \in \T$ be such that it is not a root of unity.  Define a continuous map $T : \T^2 \to \T^2$ by $T(x,y) = (x + \alpha,x + y)$, where addition is $\mod 1$. Then, the cascade $(\T^2, T)$ is  distal, minimal but not equicontinuous. Note that 
	$$T^n(x,y) = (x+n \alpha, nx +y + \frac{n(n-1)}{2} \alpha).$$
	
	We follow the constructions in \cite{ Nam} that the enveloping semigroup here,
	$$E(\T^2) \cong \T \times End(\T)$$
	
	where $End(\T)$ consists of all endomorphisms on the group $\T$.
	
	For  $f \in End(\T)$ and $\xi \in \T$, there exists a net $\{{n_\mu}\}$ of positive integers such that
	
	1) $\lim \limits_\mu n_\mu x = f(x) \ \forall x \in \T$, and
	
	2) $\lim \limits_\mu  \frac{n_\mu^2 }{2} \alpha = \xi$. 
	
	Then for $T^{n_\mu} \to p$ in $E(\T^2)$, $p(x,y) =  (x + f(\alpha),y + f(x) - f(\frac{\alpha}{2}) + \xi)$. 
	
	This gives the group operation: for $(\xi,f), (\chi,g) \in \T \times End(\T)$
	$$(\xi,f) \cdot (\chi,g) = (\xi + \chi  + f \circ g(\alpha), f+g).$$ 
	
	Also, the isomorphism $\sigma: \T \times End(\T) \to E(\T^2) $ is given as:	  
	$$\sigma(\xi,f)(x,y) = (x + f(\alpha), \xi + f(x) + y).  $$
	
	which gives the identity $\sigma(0,id) = (x+\alpha, x+y) = T(x,y)$.
	
	\bigskip
	
	Thus, for the identity $T \in E(\T^2)$:	
	
	$\  \D_T(\T \times\{0\}) =  \lbrace (x,y) \in \T^2: \exists\ \text{nets} \ \lbrace (x_{\nu},0) \rbrace \ \text{in} \ \T \times\{0\}, \ \text{and} \ \lbrace {n_\nu} \rbrace \  \text{of positive integers with} \ T^{n_\nu} \to T  \ \text{such that}  \ T^{n_\nu}(x_{\nu},0) \rightarrow (x,y) \rbrace$.
	
	Let $(x,y) \in \T^2$, then there exists a net $\{{n_\nu}\}$ of positive integers which satisfy conditions 1) and 2) stated above for $f = id$ and $\xi =0$ such that for $x_\nu =  x + \dfrac{y-x}{n_\nu} - \dfrac{n_\nu - 1}{2} \alpha$, we have $T^{n_\nu}(x_\nu,0) \to (x,y)$.
	
	$$\text{Hence} \ (x,y) \in \D_T \Longrightarrow \D_T(\T \times\{0\}) = \T^2.$$

	 and it can be seen that it gives a trivial quasifactor. Every point of $2^{\T^2}$ in the range of $\D_T$ will be an almost periodic point for $(2^{\T^2}, T_*)$, and its orbit closure a quasifactor.

	Also	$\D_T(\T \times\{0\}) \neq \T \times\{0\} $. 
	
	Thus, since $T$ is the only idempotent in $E(\T^2)$, we can see that not all points in $2^{\T^2}$ are almost periodic, i.e. $(2^{\T^2},T_*)$ is not distal. This illustrates that the induced flow of a distal, non  equicontinuous flow cannot be distal.
	
	Since $(\T^2,T)$ is distal, each quasifactor should also be distal but $(2^{\T^2},T_*)$ is not distal. Since the induced system is distal if and only if equicontinuous \cite{AN}, we note that $({\T^2},T)$ will have as quasifactors all $\X \subset2^{\T^2}$  such that $\X$ comprises of almost periodic points in $2^{\T^2}$ in a way that $(\X,T)$ is equicontinuous. Also since $(2^{\T^2},T_*)$ has the irrational rotation as the maximal equicontinuous factor, the corresponding quasifactor will be the maximal quasifactor in $(2^{\T^2},T_*)$.
	
	 A full characterization of almost periodic points of $(2^{\T^2},T_*)$ is studied in \cite{GAPP}.
\end{ex}
	
	\begin{rem}
		In the example above, $(2^{\T^2},T_*)$ is not transitive, and so $(2^{\T^2},T_*)$ is not a factor of $(E(2^{\T^2}),T_*)$. In general, for a weakly mixing $(X,T)$, $(2^X,T)$ is a factor of $(E(2^X),T)$ and the quasifactors correspond to the minimal ideals in $E(2^X)$. But minimal ideals in $E(2^X)$ project on to the minimal ideals in $E(X)$. Thus the quasifactors can be computed by locating the minimal idempotents in $E(X)$. 
	\end{rem}
	
\begin{ex}
	Recall the substitution system in Example \ref{MT}. We  look for the almost periodic points in $2^X$ here without computing $E(2^X)$. Note that $(X, \sigma)$ here is weakly mixing and so $(2^X, \sigma_*)$ will be transitive and hence $(2^X,\sigma_*)$ will be a factor of $(E(2^X),\sigma_*)$.
	
Note that the idempotents $u_1,v_1$ and $u_2,v_2$ in $E(X)$ act like some kind of duals respectively, and these are the only idempotents. Hence $\theta: J(2^X) \to J(X)$ here will be surjective. For any $A \in 2^X$, $\D_{u_1}(A)$ and $\D_{v_1}(A)$ will be off the orbits of $a, \bar{a}$ while $\D_{u_2}(A)$ and $\D_{v_2}(A)$ will be off the orbits of $b, \bar{b}$, or they will be the entire $X$ and these images will be the almost periodic points for $2^X$. Also  $(2^X, \sigma_*)$ will have transitive points, with properties as mentioned in \cite{AN}, and such points will not be almost periodic. Thus, no almost periodic point in $2^X$, other than $X$, will contain points from the orbits of both $a,\bar{a}$ and $b,\bar{b}$ respectively. Every quasifactor for $(X,\sigma)$ will have some kind of a dual quasifactor.
\end{ex}

\vspace{12pt}
\bibliography{xbib}

\end{document}